\documentclass[12pt]{amsart}
\usepackage[margin=3cm]{geometry}
\usepackage[backend=bibtex,url=false,doi=false,isbn=false,giveninits=true,style=numeric,sorting=nyt]{biblatex}
\usepackage{xspace}
\usepackage{amssymb}

\newtheorem{theorem}{Theorem}[section]
\newtheorem{lemma}[theorem]{Lemma}
\newtheorem{proposition}[theorem]{Proposition}

\newtheorem{problem}[theorem]{Problem}

\theoremstyle{definition}
\newtheorem{definition}[theorem]{Definition}

\theoremstyle{remark}
\newtheorem{remark}[theorem]{Remark}
\numberwithin{equation}{section}
\newcommand*{\Steiner}{\mathrm{S}}
\newcommand*{\T}{\top}
\newcommand*{\ResSteiner}{\mathrm{RS}}
\newcommand*{\STS}{\mathrm{STS}}
\newcommand{\KRC}{\textbf{KRC}\xspace}
\newcommand*{\KNP}{\textbf{KNP}\xspace}
\newcommand*{\BBT}{\textbf{BBT}\xspace}
\newcommand*{\BB}{\textbf{BB}\xspace}
\usepackage{enumerate}
\usepackage{booktabs}
\title{New Steiner 2-designs from old ones by paramodifications}

\author{D\'avid Mez\H{o}fi}
\address{Bolyai Institute \\
University of Szeged \\
Aradi v\'ertan\'uk tere 1\\
H-6720 Szeged, Hungary}
\email{mezofi@math.u-szeged.hu}

\author{G\'abor P. Nagy}
\address{Department of Algebra \\
Budapest University of Technology and Economics\\
Egry J\'ozsef utca 1\\
H-1111 Budapest, Hungary}
\address{Bolyai Institute \\
University of Szeged \\
Aradi v\'ertan\'uk tere 1\\
H-6720 Szeged, Hungary}
\email{nagyg@math.bme.hu}

\thanks{Support provided from the National Research, Development and Innovation Fund of Hungary, financed under the 2018-1.2.1-NKP funding scheme, within the SETIT Project 2018-1.2.1-NKP-2018-00004. Partially supported by OTKA grants 119687 and 115288.}


\date{30/04/2020}

\dedicatory{}

\begin{document}
\begin{abstract} 
Techniques of producing new combinatorial structures from old ones are commonly called trades. The switching principle applies for a broad class of designs: it is a local transformation that modifies two columns of the incidence matrix. In this paper, we present a construction, which is a generalization of the switching transform for the class of Steiner 2-designs. We call this construction paramodification of Steiner 2-designs, since it modifies the parallelism of a subsystem. We study in more detail the paramodifications of affine planes, Steiner triple systems, and abstract unitals. Computational results show that paramodification can construct many new unitals. 
\end{abstract}
\maketitle
\section{Introduction}
\label{sec:intro}

The triple $(\mathcal{P},\mathcal{B},I)$ is an \textit{incidence structure,} provided $\mathcal{P}$, $\mathcal{B}$ are disjoint sets, and $I\subseteq \mathcal{P} \times \mathcal{B}$. Using geometric language, one calls the elements of $\mathcal{P}$ points, the elements of $\mathcal{B}$ blocks, and writes $P\mathrel{I} b$ instead of $(P,b) \in I$. The incidence structure is \textit{simple,} if each block can be identified with the set of points with which it is incident. In this case, one can assume $I=\in$. For subsets $\mathcal{P}'\subseteq \mathcal{P}$ and $\mathcal{B}' \subseteq \mathcal{B}$ and $I'=I\cap (\mathcal{P}' \times \mathcal{B}')$, one has the \textit{incidence substructure} $\left( \mathcal{P}', \mathcal{B}', I' \right)$. By some abuse of notation, we may denote the latter by $\left( \mathcal{P}', \mathcal{B}', I \right)$ as well. The substructure \textit{induced by} $\mathcal{P}'\subseteq \mathcal{P}$ is defined with the set $\mathcal{B}'$ of blocks meeting $\mathcal{P}'$ in at least two points. Notice that for a substructure, a block $b\in \mathcal{B}'$ is not necessarily a subset of $\mathcal{P}'$. 

A $t$-$(n, k, \lambda)$ design, or equivalently a Steiner system $\Steiner_\lambda(t,k,n)$, is a finite simple incidence structure consisting of $n$ points and a number of blocks, such that every block is incident with $k$ points and every $t$-subset of points is incident with exactly $\lambda$ blocks. Let $\mathbf{D}=\left( \mathcal{P}, \mathcal{B}, I \right)$ be a Steiner system. The subset $\pi$ of blocks is called a \textit{parallel class,} or equivalently a \textit{$1$-factor} of $\mathbf{D}$ if it partitions the point set. If $\mathcal{B}$ is the union of disjoint $1$-factors $\pi_1,\ldots,\pi_r$, then the partition is called a \textit{$1$-factorization} and $\mathbf{D}$ is said to be \textit{resolvable.} A $1$-factorization is also called a \textit{parallelism} or a \textit{resolution.} A resolvable Steiner system $\Steiner_\lambda(t,k,n)$ is abbreviated as $\ResSteiner_{\lambda}(t,k,n)$. In general, the classification of combinatorial structures with a given set of parameters  is an old and important research topic; for details, we refer the reader to the monographs \cite{MR1192126,MR2192256,MR1729456}. Our main concern yields to designs with parameters $t=2$ and $\lambda=1$, which are called \textit{Steiner 2-designs} or \textit{linear spaces} in the literature, see \cite[Definition 2.4.9]{MR1192126}. Important classes of Steiner 2-designs are affine and projective planes of order $q$, Steiner triple systems, and abstract unitals of order $q$; the respective parameters $(n,k)$ are $(q^2,q)$, $(q^2+q+1,q+1)$, $(n,3)$ and $(q^3+1,q+1)$. 

The main result of this paper is a general construction which can produce new Steiner $2$-designs from old ones, with the same parameters. We call this construction \textit{paramodification} of $2$-designs, since it modifies the parallelism of a subsystem. Our research has been motivated by a construction of Grundh\"ofer, Stroppel and Van Maldeghem \cite{MR3533345}, which produced new abstract unitals with many translation centers, see also \cite{Moehler2020_1000117988}. As the anonymous reviewer of a previous version of this paper informed us, our construction is not completely new. In essence, Petrenjuk and Petrenjuk described it in technical reports of the University of Kirovograd (Ukraine) in the 1980s, see \cite{Petrenjuk} and its references. In particular, A. J. Petrenjuk used the method, named \textit{cut-transformations,} to construct new abstract unitals of order $3$. 

As shown in section \ref{sec:matrix}, a paramodification of a 2-$(n,k,1)$ design affects $k$ columns of the incidence matrix, all belonging to the $k$ points of a fixed block. We prove that paramodifications affecting exactly two columns are \textit{switches.} A switch or \textit{switching} is a local transformation of a combinatorial structure, which was studied for graphs, partial geometries, Steiner triples systems, codes, and other objects since the early 1980s. For the presentation of the switching principle, unification of earlier results and computational applications, see the excellent paper \cite{MR2854808} by \"Osterg\aa{}rd. In Proposition \ref{pr:noswitch}, we give a sufficient condition for a Steiner 2-design not to allow a switching. This condition implies that Hermitian unitals have no switchings, but they do have non-trivial paramodifications. 

In section \ref{sec:classes}, we study in more detail the paramodifications of affine planes, Steiner triple systems, and unitals. In the last two sections, we give an overview of the algorithmic and complexity aspects of the computation of the paramodification. We also present computational results which show that that paramodification can construct many new unitals. 

\section{Paramodification of $2$-designs}
\label{sec:paramod}

Let $\mathbf{D}=\left( \mathcal{P}, \mathcal{B}, I \right)$ be a
$t$-$(n,k,\lambda)$ design. By \cite[Theorem~1.9]{MR1729456}, the integer
\begin{equation} \label{eq:defr}
r=\lambda \frac{\binom{n-1}{t-1}}{\binom{k-1}{t-1}}=\frac{|\mathcal{B}|k}{n}
\end{equation}
is the number of blocks through a given point. The map $\chi:\mathcal{B} \to X$ is called a
\textit{proper block coloring} of $\mathbf{D}$, if for different blocks $b,b'$, $b\cap
b'\neq \emptyset$ implies $\chi(b)\neq \chi(b')$. If $|X|=m$ and $\mathbf{D}$
has a proper block coloring $\chi:\mathcal{B} \to X$ then we say that
$\mathbf{D}$ is \textit{block $m$-colorable.} 
\begin{lemma}\label{lem:nr_of_colors}
Let $\mathbf{D}=\left( \mathcal{P}, \mathcal{B}, I \right)$ be a
$t$-$(n,k,\lambda)$ design.
\begin{enumerate}[(i)]
\item Any proper block coloring of $\mathbf{D}$ needs at least $r$ colors. 
\item Any parallelism of $\mathbf{D}$ defines a block coloring with $r$ colors
when mapping each block to its parallel class. 
\item The color classes of a block coloring with $r$ colors form a parallelism
of $\mathbf{D}$. 
\item $\mathbf{D}$ is block $r$-colorable if and only if it is resolvable. 
\end{enumerate}
\end{lemma}
\begin{proof}
Since $r=|\mathcal{B}|k/n$ is the number of blocks through a point, and these blocks must
have different colors, we have (i). 
(ii) is trivial by definition. (iii) If we have $r$ colors, then for any point
$P$ and color $x$, there is a unique block on $P$ with color $x$. That is, the
color class $\chi^{-1}(x)$ is a partition of $\mathcal{P}$. (iv) follows from
(ii) and (iii). 
\end{proof}
From now on, $\mathbf{D}=\left( \mathcal{P}, \mathcal{B}, I \right)$ denotes a
Steiner 2-design on $n$ points. The incidence relation $I = \in$, that is, the blocks of $\mathbf{D}$ are subsets of size $k$ of $\mathcal{P}$. Notice that
for subsets $\mathcal{P}'\subseteq \mathcal{P}$ and $\mathcal{B}' \subseteq
\mathcal{B}$, we may consider the subsystem $\mathbf{D}'=\left( \mathcal{P}',
\mathcal{B}', I \right)$, even if an element $b'\in \mathcal{B}'$ is not a
subset of $\mathcal{P}'$. 

Fix a block $b \in \mathcal{B}$ and consider the subset 
\begin{equation}\label{eq:Cb}
C{\left( b \right)} = \left\{ b' \in \mathcal{B} \colon |b' \cap b| = 1
\right\}
\end{equation}
of blocks. We write $\mathbf{D}_b$ for the subsystem $(\mathcal{P}\setminus b,
C{(b)}, I)$. We define the map $\chi_b:C(b) \to b$ by
\begin{equation}\label{eq:chib}
\chi_b: b'\mapsto b'\cap b; 
\end{equation}
this is clearly a block coloring of $\mathbf{D}_b$. 
\begin{lemma}\label{lem:Db_resolv}
$\mathbf{D}_b$ is a resolvable $1$-$(n-k,k-1,k)$ design. %
\end{lemma}
\begin{proof}
Trivially, each block $b'\in C{(b)}$ is incident with $k-1$ point $P\in
\mathcal{P}\setminus b$, that is, $\mathbf{D}_b$ is $1$-$(n-k,k-1,k)$ design.
In $\mathbf{D}_b$, \eqref{eq:defr} implies $r=k$ and the map $\chi_b:b'\mapsto
b'\cap b$ is a block coloring with $k$ colors. By Lemma \ref{lem:nr_of_colors},
$\mathbf{D}_b$ is resolvable. 
\end{proof}
We aim to show that any parallelism of $\mathbf{D}_b$ leads to a block
design $\mathbf{D}'$ such that $\mathbf{D}$ and $\mathbf{D}'$ have the same
parameters, and they may or may not be isomorphic. To use consistent
notation, we identify the notions of a parallelism and a block coloring
with $r$ colors. 
\begin{definition}
Let $\mathbf{D}=\left( \mathcal{P}, \mathcal{B}, I \right)$ be a Steiner $2$-$(n,k,1)$ design. Let $b\in \mathcal{B}$ be a block and $\chi\colon C{(b)}\to b$ a
block coloring of the subsystem $\mathbf{D}_b$ with $k$ colors. Define the
incidence relation $I^*\subseteq \mathcal{P} \times \mathcal{B}$ by
\begin{equation}\label{eq:incstar}
P \mathrel{I^*} b' \Leftrightarrow \begin{cases} P \mathrel{I} b', &\text{if
$b'\not\in C{(b)}$ or $P\mathrel{\diagup\!\!\!\!I} b$}\\ P=\chi(b'), &\text{if $P\mathrel{I} b$ and
$b'\in C{(b)}$.} \end{cases}
\end{equation}
We call the incidence structure 
\begin{equation*}
\mathbf{D}^*=\mathbf{D}^*_{\chi,b}=(\mathcal{P}, \mathcal{B}, I^*)
\end{equation*}
the \textit{$(\chi,b)$-paramodification} of $\mathbf{D}$. 
\end{definition}
\begin{theorem}
Let $\mathbf{D}=\left( \mathcal{P}, \mathcal{B}, I \right)$ be a Steiner $2$-$(n,k,1)$ design. Let $b\in \mathcal{B}$ be a block and $\chi:C{(b)}\to b$ a
block coloring of the subsystem $\mathbf{D}_b$ with $k$ colors. Then,
$\mathbf{D}^*_{\chi,b}$ is a Steiner $2$-design with the same parameters. 
\end{theorem}
\begin{proof}
We have to show that any two points are incident with a unique block of $\mathbf{D}^*=\mathbf{D}^*_{\chi,b}$. Let $P_{1}, P_{2} \in \mathcal{P}$ be distinct points, and $\beta \in \mathcal{B}$ the unique $\mathbf{D}$-block such that $P_{1} \mathrel{I} \beta$ and $P_{2} \mathrel{I} \beta$.
\begin{enumerate}
\item $P_{1}, P_{2} \not\in b$. Then $P_{1} \mathrel{I^{*}} \beta$ and $P_{2}
\mathrel{I^{*}} \beta$ by \eqref{eq:incstar}.  Let $\gamma \in \mathcal{B}$ be a
block such that $P_{1} \mathrel{I^{*}} \gamma$ and $P_{2} \mathrel{I^{*}}
\gamma$. Then $P_{1} \mathrel{I} \gamma$ and $P_{2} \mathrel{I} \gamma$ also
by \eqref{eq:incstar}, therefore $\gamma = \beta$ as $\mathbf{D}=\left(
\mathcal{P}, \mathcal{B}, I \right)$ is a Steiner $2$-$(n,k,1)$ design.
\item $P_{1}, P_{2} \in b$. Then $\beta = b$ as $\mathbf{D}$ is a Steiner
$2$-$(n,k,1)$ design. Note that $b \not\in C{\left( b \right)}$ by the
definition of $C{\left( b \right)}$ in \eqref{eq:Cb}, hence $P_{1}
\mathrel{I^{*}} b$ and $P_{2} \mathrel{I^{*}} b$. Let $\gamma \in \mathcal{B}$ be a block such that $P_{1} \mathrel{I^{*}} \gamma$
and $P_{2} \mathrel{I^{*}} \gamma$.  If $\gamma \not\in C{\left( b \right)}$,
then $P_{1} \mathrel{I} \gamma$ and $P_{2} \mathrel{I} \gamma$ by
\eqref{eq:incstar}, therefore $\gamma = b = \beta$. If $\gamma \in C{\left( b \right)}$, then by \eqref{eq:incstar}
\begin{equation*}
\chi{\left( \gamma \right)} = P_{1} \neq P_{2} = \chi{\left( \gamma \right)},
\end{equation*}
a contradiction.
\item $P_{1} \not\in b$ and $P_{2} \in b$. In this case, $\beta \in C{\left( b
\right)}$ and $P_2 \mathrel{I^*} \beta$ if and only if $\chi(\beta)=P_2$. By Lemma \ref{lem:nr_of_colors}, $\chi$ defines a parallelism, and the color class $\chi^{-1}(P_2)$ is a parallel class in $\mathbf{D}_b$. Hence, there is a unique block $\gamma \in C{(b)}$ such that $P_1 \mathrel{I} \gamma$ and $\chi(\gamma)=P_2$. Equation \eqref{eq:incstar} implies $P_1,P_2 \mathrel{I^*} \gamma$. 
\qedhere
\end{enumerate}
\end{proof}

In general, it is not easy to determine if two paramodifications of $\mathbf{D}$ are isomorphic. We introduce the following terminology. 

\begin{definition}
The block coloring $\chi_b:C(b)\to b$, $b'\mapsto b\cap b'$ is the \textit{trivial block coloring} of the Steiner 2-design $\mathbf{D}$. Two block colorings $\chi$ and $\psi$ of $C(b)$ are said to be \textit{equivalent} if they have the same color classes. The Steiner system $\mathbf{D}$ is said to be \textit{para-rigid} if, for any block $b$, all block colorings of $\mathbf{D}_b$ are equivalent to the trivial one. 
\end{definition}

\begin{remark}
\begin{enumerate}[(i)]
\item One has $\mathbf{D}=\mathbf{D}^*_{\chi_b,b}$.
\item The block colorings  $\chi$ and $\psi$ are equivalent if there is a permutation $\pi$ of the points on $b$ such that $\psi(b')=\pi(\chi(b'))$ holds for all $b'\in C(b)$. 
\item We claim that equivalent block colorings result isomorphic paramodifications. Indeed, we can extend $\pi$ to $\mathcal{P}$ such that $\pi(P)=P$ when $P\not\in b$. Then, $\pi$ determines an isomorphism between $\mathbf{D}^*_{\psi,b}$ and $\mathbf{D}^*_{\chi,b}$. 
\item If all paramodifications of the Steiner 2-design $\mathbf{D}$ are isomorphic to $\mathbf{D}$, then we say that the paramodifications of $\mathbf{D}$ do not yield new Steiner 2-designs. Paramodifications of a para-rigid Steiner 2-design do not yield new Steiner 2-designs. The converse is not valid; see Remark \ref{rem:ag2q}.
\end{enumerate}
\end{remark}

\section{Paramodification and the incidence matrix}
\label{sec:matrix}

In this section, we describe the effect of paramodifications to the incidence matrix. 
\begin{proposition}
Let $\mathbf{D}$ be a Steiner $2$-$(n,k,1)$ design and $\mathbf{D}^*=\mathbf{D}^*_{\chi,b}$ be a
$(\chi,b)$-paramodification of $\mathbf{D}$. Let $r=(n-1)/(k-1)$. Then, the
respective incidence matrices $\mathbf{M}$ and $\mathbf{M}^*$ differ at most in
a $k \times k(r-1)$ submatrix.
\end{proposition}
\begin{proof}
Equation \eqref{eq:incstar} implies that the incidence matrices differ in the rows corresponding to the points of $b$, and in the columns corresponding to blocks in $C(b)$. Clearly, $|b|=k$ and $|C(b)|=k(r-1)$. 
\end{proof}
To have a more detailed description of the structure of the incidence matrices, consider the $n \times b$ incidence matrix $\mathbf{M}$ of the system $\mathbf{D}$ in the following way:
\begin{enumerate}
\item Let the first $k$ rows of $\mathbf{M}$ correspond to the points $P_{1},
  P_{2}, \ldots, P_{k} \in b$.
\item Let the first $r - 1$ columns of $\mathbf{M}$ correspond to the blocks in
  $C{\left( b \right)}$ incident with $P_{1}$, then let the second $r - 1$
  columns correspond to the blocks in $C{\left( b \right)}$ incident with
  $P_{2}$, and so on until $P_{k}$.
\item Right behind the columns corresponding to $C{\left( b \right)}$, put the
  column corresponding to $b$.
\item Then comes the rest of the blocks $\mathcal{B} \setminus \left( C{\left(
  b \right)} \cup b \right)$ in any order.
\end{enumerate}
The incidence matrix has the form
\begin{equation}\label{eq:incmatN}
\mathbf{M} =
\begin{pmatrix}
\mathbf{C}_{b}    & \mathbf{j}_{k}    & \mathbf{0}\\  
\mathbf{M}_{1} & \mathbf{0}_{n - k}& \mathbf{M}_{2}
\end{pmatrix},
\end{equation}
where 
\begin{equation*}
\mathbf{C}_{b} =
\begin{pmatrix}
\mathbf{j}^{\T} & \mathbf{0}^{\T} & \cdots  & \mathbf{0}^{\T}\\
\mathbf{0}^{\T} & \mathbf{j}^{\T} & \cdots  & \mathbf{0}^{\T}\\
\vdots          & \vdots          & \ddots  & \vdots\\
\mathbf{0}^{\T} & \mathbf{0}^{\T} & \cdots  & \mathbf{j}^{\T}\\
\end{pmatrix}
\end{equation*}
is a $k \times k\left( r - 1 \right)$ matrix, and $\mathbf{j}$, $\mathbf{0}$ are column vectors of length $r - 1$.

It is easy to see by the definition of $I^{*}$ in \eqref{eq:incstar}, that the
incidence matrix $\mathbf{M}^{*}$ of the new system $\mathbf{D}^{*}$ has the
form
\begin{equation*}
\mathbf{M}^{*} =
\begin{pmatrix}
\mathbf{C}^{*}_{b}    & \mathbf{j}_{k}    & \mathbf{0}\\  
\mathbf{M}_{1} & \mathbf{0}_{n - k}& \mathbf{M}_{2}
\end{pmatrix},
\end{equation*}
where except $\mathbf{C}^{*}_{b}$ all the other submatrices are the same as in
\eqref{eq:incmatN}. Hence $\mathbf{M}$ and $\mathbf{M}^{*}$ differ at most in a
$k \times k\left( r - 1 \right)$ submatrix. Finally, we notice that equivalent block colorings correspond to the permutations of the first $k$ rows of $\mathbf{M}$. 

In \cite{MR2854808}, the author defines the switching operation for constant weight codes as a transformation that concerns exactly two coordinates and keeps the studied parameter of the code unchanged. For a design $\mathbf{D}$, this means that the incidence matrix is modified in exactly two rows. As the number of 1s is constant in each column, one can interchange the 01 and 10 combinations of the two rows only. This implies the following proposition:

\begin{proposition}
Let $P,Q$ be two points of the Steiner 2-design $\mathbf{D}$. Let $b$ be the unique block on $P$ and $Q$. A switching with respect to $P$ and $Q$ is a $(\chi,b)$-paramodification. Moreover, if the block $b'\in C(b)$ is not incident with $P$ or $Q$, then it has trivial color: $\chi(b')=b\cap b'$. Conversely, a $(\chi,b)$-paramodification is a switching if and only if precisely two color classes of $\chi$ are non-trivial. \qed
\end{proposition}

In a Steiner $2$-design, a \textit{Pasch configuration} consists of six points $P_1,\ldots,P_6$ such that the triples $\{P_1,P_3,P_4\}$, $\{P_1,P_5,P_6\}$, $\{P_2,P_3,P_5\}$, $\{P_2,P_4,P_6\}$ are collinear. The design is \textit{anti-Pasch} if it does not contain any Pasch configuration. Pasch configurations are known to play an important role in switches of Steiner 2-designs. 

\begin{proposition} \label{pr:noswitch}
Let $\mathbf{D}$ be an anti-Pasch $2$-$(n,k,1)$ design. If 
\[n<2{k}^{3}-8{k}^{2}+13k-6,\]
then no switching can be carried out for $\mathbf{D}$. 
\end{proposition}
\begin{proof}
Each point is incident with $r=(n-1)/(k-1)$ blocks, and the condition is
\[(k-1)(k-2)+1>\frac{1}{2}(r-1).\]
Assume that a switching can be carried out with respect to the points $R,Q$. Let $C(Q,R)$ be the set of blocks containing precisely one of $Q$ and $R$. The $2(r-1)$ blocks are colored with two colors, say red and blue such that blocks with the same color intersect in $Q$ or $R$. As the switching is non-trivial, there are both red and blue blocks on $Q$. We can assume that at least half of the blocks on $Q$ are red. Let $a$ be a blue block on $Q$, incident with the points $Q,A_1,\ldots,A_{k-1}$. For each $i\in \{1,\ldots,k-1\}$, the block $RA_i$ is all red; let $R,A_i,P_{i1},\ldots,P_{i,k-2}$ be its points. If the points $Q,P_{is}, P_{jt}$ are collinear with $i\neq j$, then the six points $Q,R,A_i,A_j,P_{is},P_{jt}$ form a Pasch configuration. Hence, the blocks $QP_{is}$ are different for all $i\in \{1,\ldots,k-1\}$ and $s\in\{1,\ldots,k-2\}$. Moreover, $QP_{is}$ is blue since it meets the red $RA_i$. This shows that there are at least $(k-1)(k-2)+1$ blue blocks on $Q$, a contradiction. 
\end{proof}

\section{Paramodification for classes of $2$-designs}
\label{sec:classes}

In this section, we discuss the paramodification of certain well-known classes
of Steiner $2$-designs. 

\subsection{Projective and affine planes} 
The case of a finite projective plane is trivial. While the case of a finite affine plane is easy, we are not aware of any occurrence of this construction in the literature, and we give a detailed proof. 
\begin{proposition}
\begin{enumerate}[(i)]
\item Paramodifications of a finite projective plane are isomorphic. In other words, finite projective planes are para-rigid. 
\item Paramodifications of a finite affine plane are associated with the same projective plane. 
\end{enumerate}
\end{proposition}
\begin{proof}
(i) Let $\mathbf{D}$ be a projective plane of order $q$. For any line $b$, $\mathbf{D}_b$ is an affine plane of order $q$ with a unique parallelism. Hence, the proper block colorings of $C(b)$ are equivalent, and the corresponding paramodifications are isomorphic. 

(ii) Let $\mathbf{D} = \left( \mathcal{P}, \mathcal{B}, I \right)$ be an affine plane of order $q$. $\mathbf{D}$ can be embedded in a projective plane $\Pi = \left( \bar{\mathcal{P}}, \bar{\mathcal{B}}, \bar{I} \right)$ of order $q$, and $\Pi$ is unique up to isomorphism. We show that any paramodification $\mathbf{D}^*_{\chi,b}$ of $\mathbf{D}$ can be embedded in $\Pi$. This is obvious if $\chi$ and $\chi_b$ are equivalent. From now on, we assume that this is not the case, that is, there are distinct lines $\ell_1,\ell_2\in C(b)$ such that $\chi(\ell_1)=\chi(\ell_2)$ and $\ell_1\cap \ell_2 \not\in b$. Not meeting on $b$ and being disjoint off $b$, the lines $\ell_1,\ell_2$ must be parallel in $\mathbf{D}$. Take a third line $\ell_3\in C(b)$ in the same color class, $\ell_3\neq \ell_1,\ell_2$. At least one of $\ell_1\cap \ell_3$, $\ell_2\cap \ell_3$ does not lie on $b$, we must have $\ell_1\| \ell_2 \|\ell_3$. Being of the same size $q$, the color class of $\ell_1$ coincides with its parallel class. 

We claim that any color class $\kappa$ of $\chi$ is a parallel class of $\mathbf{D}$. To show this, it suffices to find two lines $m_1,m_2\in \kappa$ such that $m_1\cap m_2\not\in b$. Then, the argument above proves that $\kappa$ is indeed a parallel class. Fix $m_1\in \kappa$ and define $Q=m_1\cap b$. Let $\ell$ be the unique line which is parallel to $\ell_1$ and incident with $Q$. Then $\ell\not\in \kappa$, and therefore $\kappa$ has a line $m_2$ with is not incident with $Q$. Hence, $m_1\cap m_2\not\in b$, and the claim follows. 

Let $\ell_\infty$ be the line at infinity with respect to $\mathbf{D}$ in $\Pi$. For the (affine) point $P\in b$, let $\varepsilon(P)$ be the infinite point of the parallel class $\chi^{-1}(P)$. For $P\in \mathcal{P}\setminus b$, we put $\varepsilon(P)=P$. It is straightforward to show that $\varepsilon$ is an embedding of $\mathbf{D}^*_{\chi,b}$ in $\Pi$, which finishes the proof. 
\end{proof}

\begin{remark}\label{rem:ag2q}
Let $\mathbf{D}$ be a finite Desarguesian affine plane. While $\mathbf{D}$ is not para-rigid, it is isomorphic to any of its paramodifications. 
\end{remark}

\subsection{Steiner triple systems}
A Steiner triple system $\STS{(n)}$ is a $2$-$(n,3,1)$ design; an $\STS{(n)}$ exists if and only if $n\equiv 1,3 \pmod{6}$. Steiner triples systems, cubic graphs (regular graphs of degree $3$), and edge colorings are much connected from different points of view. For example, many recent papers deal with edge colorings of cubic graphs by Steiner triples systems, see \cite{MR3090714} and the references therein. Our approach seems to have in common with the study of cubic trades in Steiner triples systems  \cite{MR3624619}. 

Let $\mathbf{T} = \left( \mathcal{P}, \mathcal{B}, I \right)$ be an $\STS{(n)}$ and fix a triple $b=\{x,y,z\}\in \mathcal{B}$. Then, the meaning of Lemma \ref{lem:Db_resolv} is that $\mathbf{T}_b$ is a simple cubic graph whose edges can be colored by three colors.  Vizing's celebrated edge-coloring theorem asserts that any cubic graph can be edge-colored by three or four colors in such a way that adjacent edges receive distinct colors. While three colors are not enough to color all cubic graphs, and the corresponding decision problem is difficult \cite{MR635430}. Paramodifications of $\mathbf{T}$ correspond to edge $3$-colorings of $\mathbf{T}_b$. Let $\Gamma$ be an edge $3$-colored cubic graph. The union of two color classes is a regular subgraph of degree $2$; hence it is the disjoint union of cycles of even length. Let $\gamma=\{v_1,\ldots,v_{2m}\}$ be such a cycle. By switching the two colors in $\gamma$ we obtain a new edge $3$-coloring of $\Gamma$ which is equivalent to the original one if and only if $n=2m+1$. Recently, cycles in cubic graphs, their length and especially Hamiltonian cycles are a central and well-studied topic in graph theory, see \cites{MR3610759,MR3070075,MR2563136,MR2817782}. The authors of this paper are not aware of any results which could help to describe the structure of edge 3-colored cubic graphs, which occur as $\mathbf{T}_b$ for a Steiner triples system $\mathbf{T}$. 

We close this subsection by formulating an open problem on para-rigid Steiner triples systems. Notice that the Steiner triple system $\mathbf{T}$ is para-rigid, if the cubic graph $\mathbf{T}_b$ has a unique edge $3$-coloring for each block $b$.

\begin{problem}
Are there para-rigid Steiner triple systems?
\end{problem}

This problem could be tested on anti-Pasch (quadrilateral-free) Steiner triple systems, for which switching gives nothing. Anti-Pasch Steiner triple systems are very scarce, see \cite{MR1765934} and the references therein.

\subsection{Unitals with many translation centers}
The idea of the paramodification of Steiner 2-designs has been motivated by the following construction of Grundh\"ofer, Stroppel and Van Maldeghem \cite{MR3533345}. Our presentation restricts to the finite case. 

Let $q$ be an integer, $G$ a group of order $q^3-q$. Let $T$ be a subgroup of order $q$ such that conjugates $T^g$ and $T^h$ have trivial intersection unless they coincide (i.e., the conjugacy class $T^G$ forms a T.I. set). Assume that there is a subgroup $S$ of order $q+1$ and a collection $\mathcal{D}$ of subsets of $G$ such that 
\begin{enumerate}[(D1)]
\item each set $D  \in \mathcal{D}$ contains $1$,
\item any $D  \in \mathcal{D}$ has size $q+1$,
\item $|\mathcal{D}|=q-2$. 
\item For each $D  \in \mathcal{D}$, the map $(D \times D) \setminus \{(x,x) \mid x \in D\} \to G: (x,y) \to {xy}^{-1}$ is injective.
\end{enumerate}
Furthermore, we assume that the following property holds:
\begin{enumerate}
\item[(P)] The system consisting of $S\setminus \{ 1 \}$, all conjugates of $T \setminus \{1\}$ and all sets 
\[D^*:= \{ {xy}^{-1} \mid x, y  \in D, x\neq y \}\] 
with $D \in \mathcal{D}$ forms a partition of $G \setminus \{1\}$.
\end{enumerate}
We define an incidence structure with point set $\mathcal{P}=G \cup [\infty]$ and block set $\mathcal{B}=\mathcal{B}^\infty \cup \{[\infty]\}$, where
\[\mathcal{B}^\infty := \{ Sg \mid g  \in G \} \cup \{ T^h g \mid h, g  \in G \} \cup \{ Dg \mid D \in \mathcal{D}, g  \in G \}\]
and the block at infinity 
\[[\infty] = \{ T^h \mid h  \in G \}\]
consists of the conjugates of $T$ in $G$. We define two incidence relations $I$ and $I^\flat$. For both, $g\in G$ and $b\in \mathcal{B}^\infty$ are incident if and only if $g \in b$. Moreover, the points on the block at infinity $[\infty]$ are precisely the  conjugates of $T$. One defines the incidence between an affine block and a point at infinity in two different ways.
\begin{enumerate}[(a)]
\item Make each $T^h$ incident with each coset $T^{hg^{-1}} g = gT^h$ (and no other block in $\mathcal{B}^\infty$). This gives an incidence structure $\mathbb{U}_\mathcal{D} = (\mathcal{P}, \mathcal{B}, I)$. 
\item Make each conjugate $T^h$ incident with each coset $T^h g$ (and no other block in $\mathcal{B}^\infty$). This gives an incidence structure $\mathbb{U}_\mathcal{D}^\flat = (\mathcal{P}, \mathcal{B}, I^\flat)$. 
\end{enumerate}

Then both $\mathbb{U}_\mathcal{D}$ and $\mathbb{U}_\mathcal{D}^\flat$ are linear spaces and the following hold.

\begin{enumerate}[(i)]
\item $\mathbb{U}_\mathcal{D}$ and $\mathbb{U}_\mathcal{D}^\flat$ are $2$-$(q^3+1, q+1,1)$ designs; i.e., unitals of order $q$.
\item Via multiplication from the right on $G$ and conjugation on the point row of $[\infty]$, the group $G$ acts as a group of automorphisms
on $\mathbb{U}_\mathcal{D}$.
\item On $\mathbb{U}_\mathcal{D}$ the group G also acts by automorphisms via multiplication from the right on $G$ but trivially on the point row of $[\infty]$.
\item On the unital $\mathbb{U}_\mathcal{D}$ each conjugate of $T$ acts as a group of translations. Thus each point on the block $[\infty]$ is a translation center, and $G$ is two-transitive on $[\infty]$.
\item On the unital $\mathbb{U}_\mathcal{D}^\flat$ the group $G$ contains no translation except the trivial one.
\end{enumerate}

It is immediate that $\mathbb{U}_\mathcal{D}$ and $\mathbb{U}_\mathcal{D}^\flat$ are paramodifications. Indeed, the set
\[C([\infty]) = \{ T^h g \mid h, g  \in G \}\]
of blocks consists of right cosets of a conjugate of $T$, which are at the same time left cosets of another conjugate of $T$. With $b'=T^hg=gT^{hg} \in C([\infty])$, the two block colorings are
\[\chi(b')=T^h, \qquad \chi^\flat(b')=T^{hg}.\]

Starting with $G=SU(2,q)$, the subgroups $T,S$ and the system $\mathcal{D}$ can be chosen such that $\mathbb{U}_\mathcal{D}$ is isomorphic to the classical Hermitian unital of order $q$, and $\mathbb{U}_\mathcal{D}^\flat$ is isomorphic to Gr\"uning's unital \cite{MR895943}, embedded in Hall planes and their duals, see \cite[Section 3.1]{MR3533345}. In particular, Gr\"uning's unitals are paramodifications of the classical Hermitian unitals. 

In \cite{MR3533345}, the authors construct two more non-classical unitals $\mathbb{U}_\mathcal{E}$, $\mathbb{U}_\mathcal{E}^\flat$ of order $4$. In this case, $G=SU(2,4)\cong SL(2,4) \cong A_5$. Using a computer, Verena M\"ohler (Karlsruhe) \cite{Moehler2020_1000117988} found further non-classical unitals of the form $\mathbb{U}_\mathcal{D}$ and $\mathbb{U}_\mathcal{D}^\flat$ for $G=SL(2,8)$.

We finish this section with an observation on finite Hermitian unitals.

\begin{proposition}
Finite Hermitian unitals have no switchings, but they do have non-trivial paramodifications. 
\end{proposition}
\begin{proof}
As Hermitian unitals are anti-Pasch by O'Nan's result \cite[Section 3]{MR295934},  Proposition \ref{pr:noswitch} implies that finite Hermitian unitals have no switchings. However, as mentioned above, Gr\"uning's unitals are non-isomorphic paramodifications of finite Hermitian unitals.
\end{proof}

\section{Effective computation of block colorings}
\label{sec:comput}

Let $\mathbf{D}=\left( \mathcal{P}, \mathcal{B}, I \right)$ be a Steiner $2$-$(n,k,1)$ design. Let $b\in \mathcal{B}$ be a block and consider the
subsystem $\mathbf{D}_b = \left( \mathcal{P} \setminus b, C{\left( b \right)},
I \right)$. We are interested in the effective computation of all block
colorings of $\mathbf{D}_b$ to construct new Steiner 2-designs of given
parameters by paramodification. We formulate the problem in the language of
vertex colorings of simple graphs, which is known to be NP-complete in general.
However, there are methods to deal with it for certain ranges of parameters. We
compare two methods, the first one is based on clique partitions, and the other
is based on integer linear programming.

The \textit{line graph} $\Gamma=(V,E)$ of $\mathbf{D}_b$ is defined by
$V=C(b)$, and $(b_1,b_2) \in E$ if and only if $b_1$ and $b_2$ have a unique
point $P\not \in b$ in common. A straightforward consequence of
Lemma~\ref{lem:Db_resolv} is that $\Gamma$ is a $\left( k - 1
\right)^{2}$-regular simple graph. A proper block coloring $\chi
\colon C{\left( b \right)} \to b$ of the subsystem $\mathbf{D}_{b}$ is
equivalent with a proper vertex coloring of the graph $\Gamma$ using $k$
colors. We can make this equivalence more precise by using the notion of vertex
b-colorings. The latter has been introduced by Irving and Manlove \cite{MR1670155},
see also the recent survey paper \cite{MR3732606} with special emphasis on the
complexity and algorithmic aspects of computing the b-chromatic number of a
simple graph. 
\begin{definition}
Let $G=(V,E)$ be a simple graph and $\chi:V\to C$ a proper vertex coloring. The
vertex $v\in V$ is called \textit{dominant,} if for any color $c'\in
C\setminus\{\chi(v)\}$ there is a neighbor $v'$ of $v$ such that $\chi(v')=c'$.
The coloring $\chi$ is said to be a \textit{b-coloring} if there is at least
one dominant vertex in each color class. 
\end{definition}
\begin{lemma}
The map $\chi \colon C{\left( b \right)} \to b$ is a proper block coloring of
$\mathbf{D}_{b}$ if and only if it is a b-coloring of the line graph $\Gamma$
of $\mathbf{D}_{b}$.
\end{lemma}
\begin{proof}
If $\chi$ is a b-coloring of $\Gamma$, then it is also a proper block coloring of $\mathbf{D}_{b}$ trivially. Let $\chi \colon C{\left( b \right)} \to b$ be a proper block coloring of
$\mathbf{D}_{b}$ using $k$ colors. We show that each block $\beta$ is a dominant vertex of $\Gamma$. Fix a point $P \in \beta \setminus b$. By
Lemma~\ref{lem:Db_resolv}, there are precisely $k$ blocks in $C{\left( b \right)}$
incident with $P$; hence these $k$ blocks (including the block
$\beta$) form a $k$-clique in $\Gamma$. Therefore the block
coloring $\chi$ must assign different colors to these $k$ blocks, which means
that every block in the clique is dominant, and the blocks are colored with $k$
different colors. 
\end{proof}
\subsection{Colorings by the set cover method} 
One way to compute all b-colorings of the graph $\Gamma$ is to find all
solutions of a set cover problem of independent sets. In fact, a color class is an independent set of size $K=\left( n - k \right) / \left( k - 1 \right)$ and the $k$ color classes of a coloring $\chi$ are pairwise disjoint. The first step is to compute the set $Y$ of independent $K$-sets of $\Gamma$. In the second step, one constructs the graph $\Gamma^*$ with vertex set $Y$ and edges $(S_1,S_2)$ with disjoint $S_1,S_2$. In the last step, we determine all cliques of size $k$ of $\Gamma^*$. Using the \texttt{GRAPE} package~\cite{GRAPE4.8} of GAP~\cite{GAP4}, this approach is easy to implement. Moreover, \texttt{GRAPE} allows the user to exploit the
automorphism group of the Steiner 2-design $\mathbf{D}$ and the automorphism
group of the graph $\Gamma$, which makes the computation quite efficient.
\subsection{Colorings by integer linear programming}
The b-coloring problem can be formulated as an integer linear programming (ILP)
problem~\cite[Section 8.4]{MR3732606}, for an exact formulation
see~\cite[Section~2]{MR3958743}. Most of the ILP solvers are optimized to find one solution to each problem. However, for our block coloring problem, we are interested in finding all solutions. Up to our knowledge, this is only possible with the MILP solver SCIP~\cite{GleixnerEtal2018OO}. 

As mentioned above, there are many ways to give the ILP formulation of a graph coloring problem. The assignment-based model~\cite[Subsection~2.2]{MR3786990} is the standard formulation of the vertex coloring problem. This formulation uses only binary variables, one for each color and one for each vertex-color pair, and the objective is to minimize
the number of used colors. Since we are only interested in $k$-colorings, this allows us
to simplify the model slightly.

There are other approaches as well, based on partial ordering, like POP and
POP2~\cite[Section~3]{MR3786990}. The idea is to introduce a partial ordering
on the union of the vertices and the color set, and encode these relations with
binary variables. The authors also provide the relation between these new
variables and the variables occurring in the standard assignment-based model.

A drawback of the ILP formulations is that, in contrast to the set cover
method, it is hard to make use of the symmetry of the underlying graph. We conclude that since \texttt{GRAPE} is very efficient in coping with symmetries of a
line graph, it is better suited to compute all paramodifications of a given
Steiner 2-design.
\section{Paramodification of unitals of orders 3 and 4}
\label{sec:unitals}

In this section we present computational results on paramodifications of known
small unitals. In this way we construct $173$ new unitals of order $3$, and
$25\,641$ new unitals of order $4$. We study the following classes of abstract
unitals of order at most $6$:

\begin{description}%
\item[Class BBT] 909 unitals of order 3 by Betten, Betten and
  Tonchev~\cite{MR1991559}.
\item[Class KRC] $4\,466$ unitals of order 3 by Kr\-\v{c}a\-di\-nac \cite{MR2194756}. 
  This class contains all abstract unitals of order $3$ with a non-trivial
  automorphism group. 722 of the \BBT unitals appear in \KRC. 
\item[Class KNP] 1777 unitals of order 4 by Kr\v{c}adinac, Naki\'c and
  Pav\v{c}evi\'c~\cite{MR2838908},
\item[Class BB] two cyclic unitals of orders 4 and 6 by Bagchi and
  Bagchi~\cite{MR1008159}. The cyclic \BB unital of order $4$ is contained in \KNP, as well.
\end{description}

We access the libraries of small unitals and carry out the computations using
the GAP package \texttt{UnitalSZ} \cite{UnitalSZ0.6}. If $\mathbf{D}$ is a \BB
unital of order $6$, then $\mathbf{D}_b$ has a unique block coloring for each
block $b$; that is, paramodification gives no new unitals of order $6$.

The \emph{paramodification graph} $\Psi_q$ for a given order $q$ consists of
one vertex for each equivalence class of unitals of order $q$ and with edges
between two vertices whenever one can get from one equivalence class to the
other via a paramodification. As paramodifications are reversible, we may
consider undirected graphs. The connected components of the paramodification
graph are called \emph{paramodification classes}. Paramodification graphs are defined analogously to \textit{switching graphs} in~\cite{MR2854808}. 

We carried out computations to determine the paramodification classes of
$\Psi_3$ and $\Psi_4$, containing at least one unital from the classes \BBT,
\KRC or \KNP.  For the case of order $3$, we found all such classes, resulting
$173$ new unitals of order $3$. This subgraph of $\Psi_3$ is complete in the
sense that all paramodifications of all vertices are known, see Table
\ref{tbl:comps}.

Consider the switching graph on the unitals from the classes \BBT, \KRC, and the
newly found 173 paramodifications of them. As switches are special cases of
paramodifications, this switching graph is a subgraph of the graph mentioned
above. By restricting the type of transformations to switches, we lose 623 edges
between the unitals in contrast to paramodifications, and only 131 of the new
173 unitals are reachable via switching. In the paramodification subgraph, there
are 3182 isolated vertices according to Table~\ref{tbl:comps}; in the switching
graph, this number is 3525.

In the case of order $4$, out of the $1\,777$ unitals of \KNP, $1\,458$ turn
out to be isolated vertices of $\Psi_4$. By repeating the paramodification
step, we produced $25\,641$ new unitals of order $4$. However, the graph is
incomplete as it has unfinished vertices; these are unitals whose
paramodifications have not been computed yet. Not counting the isolated
vertices, the number of complete paramodification classes is $142$. The remaining
$6$ classes are all incomplete, with $12\,610$ unfinished vertices in total.
Concerning the growth of the connected components, it is hard to say
anything mathematically reasonable. The largest component with $7\,596$ known
vertices has $8$ vertices of \KNP type, and its growth in the breadth-first
search is \[8,\quad 45,\quad 425,\quad 7118,\quad ???\]

In Table \ref{tbl:runtimes}, we present the comparison of run-times of different
algorithms for the computation of $(\chi,b)$-paramodifications. The reader can
find further scientific data on the paramodification of unitals on the web page
\url{https://davidmezofi.github.io/unitals/}.

\begin{table}
\caption{Distribution of the sizes of the paramodification classes\label{tbl:comps}}
\begin{tabular}{rrr}
\toprule
size of class & nr of classes in $\Psi_3$  & nr of classes in $\Psi_4$ \\
\midrule
isolated vertex & 3\,182 & 1\,458\\
2--5 & 466 & 99\\
6--10 & 35 & 13\\
11--100 & 13 & 16\\
101-- 1\,000 && 14\\
1\,342 && $1^*$\\
1\,478 && $1^*$\\
2\,557 && $1^*$\\
3\,487 && $1^*$\\
4\,035 && $1^*$\\
7\,596 && $1^*$\\
\bottomrule
\end{tabular}
\end{table}

\begin{table}
\caption{Mean and maximal run-times of different methods in milliseconds of 30
random KNP unitals and a random block\label{tbl:runtimes}}
\begin{tabular}{lrr}
\toprule
method & mean & maximum\\
\midrule
set cover (GAP) & 142 & 316\\
assignment (SCIP) & 3369 & 9804\\
POP (SCIP) & 4082 & 12266\\
POP2 (SCIP) & 4444 & 14707\\
\bottomrule
\end{tabular}
\end{table}

\printbibliography

\end{document}